\documentclass
[11pt]{article} \oddsidemargin 0in \evensidemargin 0in \textheight
9in \textwidth 6.5in \topmargin 0in \headheight 0in
\usepackage{amsmath,amsfonts,amsthm,amssymb}

\newtheorem{thm}{Theorem}

\newtheorem{defn}{Definition}

\newtheorem{rem}{Remark}
\newtheorem{exple}{Example}

\title{Rectangular metric like type spaces and fixed points}

\author {N. Mlaiki$^{1}$, K. Abodayeh$^{2},$ T. Abdeljawad$^{3},$ M. Abuloha $^{4}$
   \\ {\it $^{1, 2, 3}$Department of Mathematical Sciences, Prince Sultan University}
   \\ {\it Riyadh, Saudi Arabia}
   \\ E-mail:  $^{1}$ nmlaiki@psu.edu.sa\\
               $^{2}$ kamal@psu.edu.sa\\
               $^{3}$ tabdeljawad@psu.edu.sa\\
               $^{4}$ \it Palestine Technical University Kadooreei Tulkarim Palestine\\
                email: muhib.abuloha@ptuk.edu.ps}

\date{}

\begin{document}
\maketitle{}
\begin{abstract}
In this paper we introduce the concept of the rectangular metric like spaces, along with its topology and
we prove some fixed point theorems under different contraction principles. We introduce the concept of modified metric-like space as well and prove some topological and convergence properties under the symmetric convergence. Some examples are given to illustrate the proven results and enrich the new introduced metric type spaces.

\end{abstract}
\footnotetext {2010 Mathematics Subject Classification: 54E15, 54E50, 47H10}

\footnotetext{Keywords: Rectangular metric like, modified metric type space, fixed point, contraction, symmetric convergence.}

\section{Introduction}
The generalization of Banach contraction principle, which  has many applications in different branches of science and engineering, depends either on generalizing the metric type space or the contractive type mapping ( see \cite{Kirk Shahzad} and the references therein). The generalization of the metric type space based on reducing or modifying the metric axioms. We name for example quasi-metric spaces, partial metric spaces, $m-$metric spaces, $m_b-$metric spaces, $S_p$ metric spaces,  rectangular metric spaces, $b-$metric spaces, metric-like spaces and rectangular partial metric spaces and so on (\cite{Amini 2012}-\cite{ Asadi}). In fact, losing or weakening some of the metric axioms causes the loss of some metric type convergence properties and hence brings obstacles in proving some fixed point theorems. These obstacles force researchers to develop the techniques in proving their fixed point results which leads  to development in fixed point theory. Consequently, the new obtained  fixed point results will be valid for more applications of modelling problems in different areas  where fixed point techniques are necessary.
In this article, we  restrict ourselves on developing metric-like spaces by introducing modified metric-liked spaces, rectangular metric-like spaces and  rectangular modified metric-like spaces and we shall prove some fixed point theorems in  rectangular metric-like spaces. Examples will be given to support our results and the symmetric convergence will be studied in the newly introduced metric type spaces.

\section{Partial metric and rectangular metric preliminaries}
\begin{defn}\cite{Mathews} (partial metric space)
Let $X$ be a nonempty set. A mapping $p:x\times X\rightarrow \mathbb{R}^+$ is said to be a partial metric on $X$ if for any $x,y,z \in X$, it satisfies the following conditions:
\begin{enumerate}
  \item ($p_1$) $x=y$ if and only if $p(x,y)=p(x,x)=p(y,y)$;
  \item ($p_2$) $p(x,x)\leq p(x,y)$;
  \item ($p_3$)  $p(x,y)=p(y,x)$;
  \item ($p_4$) $p(x,y)\leq p(x,z) + p(z,y)-p(z,z)$.
\end{enumerate}
In this case the pair $(X,p)$ is called a partial metric (PM) space.
\end{defn}

\begin{defn}\cite{Branc}\label{def3} (rectangular metric space (Branciari metric space))
Let $X$ be a nonempty set. A mapping $d:x\times X\rightarrow \mathbb{R}^+$ is said to be a rectangular metric on $X$ if for any $x,y \in X$ and all distinct points $u,v \in X\setminus \{x,y\}$, it satisfies the following conditions:
\begin{enumerate}
  \item ($R_1$) $x=y$ if and only if $d(x,y)=0$;
  \item ($R_2$) $d(x,y)=d(y,x)$;
  \item ($R_3$) $d(x,y)\leq d(x,u) + d(u,v)+d(v,y)$.
\end{enumerate}
In this case the pair $(X,d)$ is called a rectangular metric (RM) space.
\end{defn}

In \cite{Satish} the notion of rectangular metric space was extended to rectangular partial metric spaces as follows.

\begin{defn}\cite{Satish} (rectangular partial metric space )
Let $X$ be a nonempty set. A mapping $\rho:x\times X\rightarrow \mathbb{R}^+$ is said to be a rectangular metric on $X$ if for any $x,y \in X$ and all distinct points $u,v \in X\setminus \{x,y\}$, it satisfies the following conditions:
\begin{enumerate}
  \item ($RP_1$) $x=y$ if and only if $\rho(x,y)=\rho(x,x)=\rho(y,y)$;
  \item  ($RP_2$) $\rho(x,x)\leq \rho(x,y)$;
  \item ($RP_3$) $\rho(x,y)=\rho(y,x)$;
  \item ($RP_4$) $\rho(x,y)\leq \rho(x,u) + \rho(u,v)+\rho(v,y)-\rho(u,u)-\rho(v,v)$.
\end{enumerate}
In this case the pair $(X,\rho)$ is called a rectangular partial  metric (RPM)  space.
\end{defn}

It is clear that every rectangular metric space is a rectangular partial  metric space but the converse is not true.

\begin{exple}\cite{Satish}
   Let $𝑋 = [0, a]$ and $\alpha\geq a\geq 3$. Define the mapping $\rho:X\times X\rightarrow \mathbb{R}^+$ by
$$
\rho (x,y)=
\left\{
\begin{array}{ll}
  x & \mbox{if $x=y$} \\
 \frac{3\alpha +x+y}{2} & \mbox{if $x,y\in\{1,2\} ,x\neq y$} \\
  \frac{\alpha +x+y}{2} & \mbox{otherwise}
\end{array}
\right. .
$$
Then, $(X�, \rho�)$ is a rectangular partial metric space,but it is not a rectangular metric space,
because for any $x>0$, we have $\rho (x,x)=x\neq 0.$
\end{exple}

For convergence , completeness and examples  of RM, PM and  RPM spaces we refer to \cite{Branc,Mathews,Satish}. In general, for metric type spaces where the self-distance need not be zero we use the convergence
$$x_n\rightarrow x \Leftrightarrow \varrho(x_n,x)=\varrho(x,x),$$ where $\varrho$ is metric type function under which the self distance may not be zero.

\section{Metric-like spaces, modified metric-like spaces and symmetric convergence}

\begin{defn}\cite{Amini 2012}\label{def4}
Let $X$ be a nonempty set. A mapping $\sigma:x\times X\rightarrow \mathbb{R}^+$ is said to be a metric-like on $X$ if for any $x,y,z \in X$, it satisfies the following conditions:
\begin{itemize}

   \item ($\sigma_1$) $\sigma(x,y)=0$ implies $x=y$;
  \item ($\sigma_2$) $\sigma(x,y)=\sigma(y,x)$;
  \item  ($\sigma_3$) $\sigma(x,y)\leq \sigma(x,z) + \sigma(z,y)$.
\end{itemize}
In this case the pair $(X,\sigma)$ is called a metric- like space ($ML$-space).
\end{defn}

Every metric-like space is a topological space whose topology is generated by the base consisting of the open $\sigma-$balls
$$B_\sigma(x,\delta)=\{y \in X: |\sigma(x,y)-\sigma(x,x)|<\delta\}, ~~x \in X,~ \delta>0.$$
Note the difference between  the balls $B_\sigma(x,\delta)$ and the balls $B_p(x,\delta)$, which is due to the absence of the smallness of the self distance condition $(p_2)$ from the metric-like. Also, since self distance need not be zero in metric-like spaces then convergence and completeness in metric-like spaces still resembles that in partial metric spaces. Indeed, a sequence $\{x_n\}$ in a metric-like space converges to a point $x \in X$ if and only if $\lim_{n\rightarrow \infty}\sigma(x_n,x)=\sigma(x,x)$ and the sequence $\{x_n\}$ is called $\sigma-$Cauchy if $\lim_{m,n\rightarrow \infty}\sigma(x_n,x_m)$ exists and is finite. The metric-like space $(X,\sigma)$ is called complete if for each $\sigma-$Cauchy sequence $\{x_n\}$ there exists $x \in X$ such that
$$\lim_{n\rightarrow \infty }\sigma(x_n,x)=\sigma(x,x)=\lim_{m,n\rightarrow \infty}\sigma(x_n,x_m).$$
\begin{rem}\label{remark1}

Metric-like spaces lose some topological and convergence properties that metric space can have. For example
\begin{itemize}
  \item Limits are not unique in $ML-$ spaces. Take $X=\{a,b\}$ and let $\sigma(x,y)=1$ for any $x \in X$. Then, clearly the sequence $x_n=1$ for all $n$ converges to both $a$ and $b$. Notice that $\sigma(a,a)=\sigma(b,b)=1\neq 0$. However, if $x_n\rightarrow x$ and $x_n\rightarrow y$ such that $x,y \in \Lambda=\{z \in X:\sigma(z,z)=0\}$ then $\sigma(x,y)\leq \sigma(x_n,x)+\sigma(x_n,y)$ and hence by letting $n\rightarrow \infty $ we conclude that $\sigma(x,y)=0$ and hence $x=y$.
  \item $y \in B_\sigma(x,\delta)$ does not necessarily imply that $y \in B_\sigma(x,\delta)$.
  \item Convergent sequences are not necessarily $\sigma-$Cauchy.
  \item If $x_n$ is a $\sigma-$Cauchy sequence in $X$ and has a convergent subsequence $x_{n_i}$ to $x$, then not necessarily $x_n\rightarrow x$.
  \item If $\{x_n\}$ and  $\{x_n\}$ are $\sigma$Cauchy sequences in $X$ then it is not necessary that $\lim_{n\rightarrow\infty} \sigma(x_n,y_n)$ exists.
  \item If $x_n\rightarrow x$ and $y_n\rightarrow y$ in $(X,\sigma)$ then it not necessarily that $\lim_{n\rightarrow\infty}\sigma(x_n,y_n)=\sigma(x,y)$.
\end{itemize}
\end{rem}
Upon Remark \ref{remark1} above we define the following modified metric-like space ($mML$ space).
\begin{defn} (modified metric-like spaces)
Let $X$ be a nonempty set. A mapping $\sigma_m:x\times X\rightarrow \mathbb{R}^+$ is said to be a modified metric-like on $X$ if for any $x,y,z \in X$, it satisfies the following conditions:
\begin{itemize}

   \item ($m\sigma_1$) $\sigma_m(x,y)=0$ implies $x=y$;
  \item ($m\sigma_2$) $\sigma_m(x,y)=\sigma_m(y,x)$;
  \item  ($m\sigma_3$) $\sigma_m(x,y)\leq \sigma_m(x,z) + \sigma_m(z,y)-\sigma_m(z,z)$.
\end{itemize}
In this case the pair $(X,\sigma_m)$ is called a modified  metric- like space ($mML$-space).
\end{defn}
It is clear that every partial metric space is $mML-$space and every $mML-$space is $ML-$space.

\begin{exple} (An $ML-$space which is not $mML-$spaces)
  let $X=\{1,2,3,4,5\}$ and define the mapping $d:X\times X\rightarrow \mathbb{R}^+$ such that:
  $d(x,y)=2$ for all $x\neq y$, $d(x,x)=0$ for all $x\neq 1$ and $d(1,1)=3$. Then it is clear that the conditions $\sigma_1 ,\sigma_2$ of Definition \ref{def4} are satisfied. We need to verify the last condition. If $x\neq y$, then we have
  $d(x,u)+d(u,y)=2+2\geq d(x,y)$. Also, if $x=y=1$, then $d(1,u)+d(u,1)=2+2\geq d(1,1)$. Finally, if $x=y\neq 1$,
  $d(x,u)+d(u,x)\geq 0=d(x,x)$. Therefore, $(X,d)$ is a $ML$-space but it is not a $mML-$space because
  $d(2,1)+d(1,3)-d(1,1)=2+2-3=1\leq d(2,3)$.
\end{exple}

\begin{defn} (symmetric convergence in  metric-like spaces)
We shall say that a sequence $\{x_n\}$ of a  metric-like space $(X,\sigma)$ is symmetric convergent to $x \in X$ if for every $\epsilon>0$
there exists $n_0 \in \mathbb{N}$ such that for each $n\geq n_0$ we have  $$x_n \in B_\sigma(x,\epsilon) ~\texttt{and}~x \in B_\sigma(x_n,\epsilon). $$
Equivalently, if

\begin{equation}\label{eqi}
\lim_{n\rightarrow \infty}\sigma(x_n,x)=\lim_{n\rightarrow \infty}\sigma(x_n,x_n)=\sigma(x,x).
\end{equation}

\end{defn}
We shall denote $x_n\rightarrow^s x$ for symmetric convergence which is characterized by (\ref{eqi}). It is clear that symmetric convergence implies $\sigma-$convergence or the $\sigma-$ topology  convergence.
\begin{thm}Let $(X,\sigma_m)$ be a $mML-$space. Then
\begin{enumerate}
  \item If $x_n\rightarrow^s x$ then $\{x_n\}$ is $\sigma_m-$Cauchy.
  \item If $\{x_n\}$ is $\sigma_m-$Cauchy and  has a subsequence $\{x_{n_i}\}$ such that $x_{n_i}\rightarrow^s x$ then $x_{n}\rightarrow^s x$.
  \item If $\{x_n\}$ and $\{y_n\}$ are $\sigma_m-$Cauchy sequences then $\lim_{n\rightarrow\infty} \sigma_m(x_n,y_n)$ exists.
  \item If $x_{n}\rightarrow^s x$ and $y_{n}\rightarrow^s y$ then $\lim_{n\rightarrow\infty} \sigma_m(x_n.y_n)=\sigma(x,y)$.
\end{enumerate}

\end{thm}
\begin{proof}

\end{proof}
\begin{enumerate}
  \item Assume $x_n\rightarrow^s x$.  Then $\lim_{n\rightarrow\infty} \sigma_m(x_n.y_n)=\sigma(x,y)$. By $m\sigma_3$, for each $l,n \in \mathbb{N}$  we have $$\sigma_m(x_n,x_l)\leq \sigma_m(x_n,x)+\sigma_m(x,x_l)-\sigma_m(x,x) $$
      and $$ \sigma_m(x,x)\leq \sigma_m(x,x_n)+\sigma_m(x_n,x_l)+ \sigma_m(x_l,x)-\sigma_m(x_n,x_n)-\sigma_m(x_l,x_l).$$
      Let $l,n\rightarrow\infty$. Then $\sigma_m(x,x)\leq \lim_{n,l\rightarrow\infty}\sigma_m(x_l,x_n)\leq \sigma_m(x,x)$. Hence $\lim_{n,l\rightarrow\infty}\sigma_m(x_l,x_n)= \sigma_m(x,x)$ and so $\{x_n\}$ is $\sigma_m-$Cauchy.
  \item Let $\{x_n\}$ is $\sigma_m-$Cauchy and  has a subsequence $\{x_{n_i}\}$ such that $x_{n_i}\rightarrow^s x$. Then $$\lim_{i\rightarrow\infty }\sigma_m(x_{n_i},x)=\lim_{i\rightarrow\infty }\sigma_m(x_{n_i},x_{n_i})=\sigma_m(x,x).$$ Since $\{x_n\}$ is $\sigma_m-$Cauchy then there exists $r>0$ such that $\lim_{l,n\rightarrow \infty} \sigma_m(x_n,x_l)=r$. It is clear that $\sigma_m(x,x)=r$ as well. On the other hand by ($m\sigma_3$) we have $$\sigma_m(x_n,x)\leq \sigma_m(x_n,x_{n_i})+\sigma_m(x_{n_i},x)-\sigma_m(x_{n_i},x_{n_i}),$$ and
      $$\sigma_m(x_{n_i},x)\leq \sigma_m(x_{n_i},x_n)+\sigma_m(x_n,x)-\sigma_m(x_{n},x_n).$$
      Therefore, $$\sigma_m(x_{n_i}, x)-\sigma_m(x_{n_i}, x_n)+\sigma_m(x_{n}, x_n)\leq \sigma_m(x_{n}, x)\leq \sigma_m(x_n,x_{n_i})+\sigma_m(x_{n_i}, x)-\sigma_m(x_{n_i},x_{n_i}).$$
      If we let $n,i\rightarrow \infty$ then $\sigma_m(x,x)-r+r\leq \lim_{n,l\rightarrow\infty} \sigma_m(x_n,x_l)\leq r+\sigma_m(x,x)-\sigma_m(x,x)$. From which it follows that $\lim_{n\rightarrow\infty}\sigma_m(x_n,x)=\lim_{n\rightarrow\infty}\sigma_m(x_n,x_n)=\sigma_m(x,x)$ and hence $x_n\rightarrow ^s x$.
  \item Assume  $\{x_n\}$ and $\{y_n\}$ are $\sigma_m-$Cauchy sequences in $X$. Then, there exist $r_1,r_2>0$ such that
   $\lim_{n,l\rightarrow\infty}\sigma_m(x_n,x_l)=r$ and $\lim_{n,l\rightarrow\infty}\sigma_m(y_n,y_l)=r_2$ . It is sufficient to prove that the sequence $\{\sigma_m(x_n,y_n)\}$ is Cauchy in $\mathbb{R}$. By ($m\sigma_3$) for each $n,l \in \mathbb{N}$ we have
   $$\sigma_m(x_n,y_n)\leq\sigma_m(x_n,x_l)+\sigma_m(x_l,y_l)+\sigma_m(y_l,x_n)-\sigma_m(x_l,x_l)-\sigma_m(y_l,y_l),$$
   and
   $$\sigma_m(x_l,y_l)\leq \sigma_m(x_l,x_n)+\sigma_m(x_n,y_n)+\sigma_m(y_n,y_l)-\sigma_m(x_n,x_n)-\sigma_m(y_n,y_n).$$
   From which it follows that
   $$\sigma_m(x_n,x_n)+\sigma_m(y_n,y_n)-\sigma_m(x_l,x_n)-\sigma_m(y_n,y_l)\leq \sigma_m(x_n,y_n)-\sigma_m(x_l,y_l)\leq $$
   $$\sigma_m(x_n,x_l)+\sigma_m(y_l,y_n)-\sigma_m(x_l,x_l)-\sigma_m(y_l,y_l).$$ Let $n,l\rightarrow\infty$ then
    $$r_1+r_2-r_1-r_2\leq \lim_{n,l\rightarrow\infty}( \sigma_m(x_n,y_n)-\sigma_m(x_l,y_l)\leq r_1+r_2-r_1-r_2.$$
    Hence $|\sigma_m(x_n,y_n)-\sigma_m(x_l,y_l)|=0$ and so $\{\sigma_m(x_n,y_n)\}$ is Cauchy in $\mathbb{R}$.
  \item assume $x_{n}\rightarrow^s x$ and $y_{n}\rightarrow^s y$. Then $$\lim_{n\rightarrow\infty}\sigma_m(x_n,x)=\lim_{n\rightarrow\infty}\sigma_m(x_n,x_n)=\sigma_m(x,x),$$
      and
     $$\lim_{n\rightarrow\infty}\sigma_m(y_n,y)=\lim_{n\rightarrow\infty}\sigma_m(y_n,y_n)=\sigma_m(y,y).$$
     Now for each $n \in \mathbb{N}$ we have  $$\sigma_m(x_n,y_n)\leq \sigma_m(x_n,x)+\sigma_m(x,y)+\sigma_m(y,y_n)-\sigma_m(x,x)-\sigma_m(y,y), $$
     and
     $$\sigma_m(x,y)\leq \sigma_m(x,x_n)+\sigma_m(x_n,y_n)+\sigma_m(y_n,y)-\sigma_m(x_n,x_n)-\sigma_m(y_n,y_n). $$
     Finally, letting $n\rightarrow\infty$ will lead to
     $$\lim_{n\rightarrow\infty}\sigma_m(x_n,y_n)\leq \sigma_m(x,y)\leq\lim_{n\rightarrow\infty}\sigma_m(x_n,y_n), $$ and thus
     $\lim_{n\rightarrow\infty}\sigma(x_n,y_n)=\sigma(x,y)$.
\end{enumerate}

\section{Rectangular metric-like and rectangular modified metric-like topological spaces}
In this section we introduce new concept of rectangular metric-like and rectangular modified metric-like spaces.

\begin{defn}\label{def7}
Let $X$ be a nonempty set and $\rho_{r}:X^{2}\rightarrow [0,\infty)$ be a function. If
the following conditions are satisfied for all $x,y$ in $X$
\begin{enumerate}
\item $\rho_{r}(x,y)=0 \Rightarrow x=y$
\item $\rho_{r}(x,y)= \rho_{r}(y,x)$.
\item $\rho_{r}(x,y)\le \rho_{r}(x,u)+\rho_{r}(u,v)+\rho_{r}(v,y), \ \text{for all} \ u,v\in X\setminus\{x,y\}$
\end{enumerate}
then the pair $(X,\rho_{r})$ is called a rectangular metric-like (RML) space.
\end{defn}

\begin{defn}
Let $X$ be a nonempty set and $\rho_{mr}:X^{2}\rightarrow [0,\infty)$ be a function. If
the following conditions are satisfied for all $x,y$ in $X$
\begin{enumerate}
\item $\rho_{rm}(x,y)=0 \Rightarrow x=y$
\item $\rho_{rm}(x,y)= \rho_{rm}(y,x)$.
\item $\rho_{rm}(x,y)\le \rho_{rm}(x,u)+\rho_{rm}(u,v)+\rho_{rm}(v,y)-\rho_{rm}(u,u)-\rho_{rm}(v,v), \ \text{for all} \ u,v\in X\setminus\{x,y\}$
\end{enumerate}
then the pair $(X,\rho_{rm})$ is called a rectangular modified metric-like (RMML) space.
\end{defn}

\begin{exple}
  Let $X=\{1,2,3,4,5\}$ and define the mapping $\rho_r : X^2 \rightarrow [0,\infty)$ by
  $$
  \rho_r (x,y)=
 \left\{
\begin{array}{ll}
  2.5 & \mbox{for $x\neq y$} \\
 5 & \mbox{if $x=y=1$} \\
  0 & \mbox{otherwise}
\end{array}
\right. .
$$
Then, it is clear that conditions 1 and 2  of Definition \ref{def7} are satisfied.
We need to verify the last condition of the definition. For all $u,v\in X\setminus\{x,y\}$, we have
$\rho_r (x,u)+\rho_r (u,v)+\rho_r (v,y)=2.5+\rho_r (u,v)+ 2.5=5+\rho_r (u,v)\geq \rho_r (x,y)$,
for all $x,y\in X$. Therefore, $(X,\rho_r )$ is a RML-space but it is not a RMML-space because
$\rho_r (2,1)+\rho_r (1,1)+\rho_r (1,3)-\rho_r (1,1)-\rho_r (1,1)=2.5+5+2.5-5-5=0\leq \rho_r (2,3)=2.5$.
Moreover, the space $(X,\rho_r )$ is not a rectangular partial metric because of the condition $RP_4$ of Definition \ref{def3}.
\end{exple}

\begin{exple}
  Let $(X,\rho_r )$ be a RMML-space. Then the space $(X,\rho_{rm} )$ is also a RMML-space, where $\rho_{rm} (x,y)=\rho_r (x,y)+\alpha $ and $\alpha >0$.
  To prove this argument, we need to prove the triangle inequality of the definition.
For any $x,y\in X$ and $u,v\in X\setminus \{x,y\}$, we have
$
\rho_{rm} (x,u)+\rho_{rm}(u,v)+\rho_{rm}(v,y)-\rho_{rm}(u,u)-\rho_{rm}(v,v)=\rho_{r}(x,u)+\rho_{r}(u,v)+\rho_{r}(v,y)-\rho_{r}(u,u)-\rho_{r}(v,v)+\alpha \geq \rho_{r}(x,y)+\alpha  =\rho_{rm}(x,y).
$
\end{exple}

\begin{exple}
  Let $X=(0,1)$ and define the mapping $\rho_{mr}:X^{2}\rightarrow [0,\infty)$ by $\rho_{mr} (x,y)=|x-y|+2$. Then $(X,\rho_{rm} )$ is a RMML-space.
We need to verify the triangle inequality. For any $x,y\in X$ and $u,v\in X\setminus \{x,y\}$, we have
$
\rho_{rm} (x,u)+\rho_{rm}(u,v)+\rho_{rm}(v,y)-\rho_{rm}(u,u)-\rho_{rm}(v,v)=|x-u|+|u-v|+|v-y|+2 \geq  |x-y|+2  =\rho_{rm}(x,y).
$

\end{exple}

\begin{defn}\label{d9}
\begin{enumerate}
\item A sequence $\{x_{n}\}$ is called $\rho_{r}-$convergent ($\rho_{rm}-$convergent) in a rectangular metric like space $(X,\rho_{r}),$ (a rectangular modified metric like space $(X,\rho_{rm}),$) if there exists $x\in X$
such that $\lim_{n\rightarrow \infty}\rho_{r}(x_{n},x)=\rho_{r}(x,x).$ ($\lim_{n\rightarrow \infty}\rho_{rm}(x_{n},x)=\rho_{r}(x,x).$)\\
\item A sequence $\{x_{n}\}$ is called $\rho_{r}-$Cauchy if and only if $\lim_{n,m\rightarrow \infty}\rho_{rm}(x_{n},x_{m})$

($\lim_{n,m\rightarrow \infty}\rho_{rm}(x_{n},x_{m})$ )
exists and finite.\\
\item A rectangular metric like space $(X,\rho_{r})$ (rectangular modified metric like $(X,\rho_{rm})$ )is called $\rho_{r}-$complete ( $\rho_{rm}-$complete) if every $\rho_{r}-$Cauchy
($\rho_{rm}-$Cauchy) sequence is $\rho_{r}-$convergent ($\rho_{rm}-$convergent).
\end{enumerate}
\end{defn}
\begin{rem}
The convergence defined above in Definition \ref{d9} is the convergence obtained in the sense of the topology generated by the open balls
$B_\varrho(x,\delta)=\{y \in X: |\varrho(x,y)-\varrho(x,x)|<\delta\},~~x \in X,~~\varrho \in \{\rho_r,\rho_{rm}\}$. This convergence is weaker than the symmetric convergence discussed before.
\end{rem}
\begin{defn} (Continuity of maps) A mapping $f:(X,\varrho_1)\rightarrow
 (Y,\varrho_1)$ between two metric type spaces is continuous at $x \in X$ if and only if $f(x_n)\rightarrow^{\varrho_2} f(x)$ whenever $\{x_n\}$ is a sequence in $X$ such that
$x_n\rightarrow^{\varrho_1} x$. For example if $(X,\varrho_1)$ and $(Y,\varrho_2)$ are metric type spaces in which the self distance is not necessary zero (such as ML, MML, PM, RPM, RML, RMML spaces)then  $f:X\rightarrow Y$ is continuous at $x \in X$ if and only if $\varrho_2(f(x_n),f(x))\rightarrow \varrho_2(f(x),f(x))$ whenever $\varrho_1(x_n,x)\rightarrow \varrho_1(x,x)$. For other types of continuity when convergence is varying between symmetric convergence and topology-convergence we refer to \cite{KNTW JMA 2016}.
\end{defn}
\section{Rectangular metric-like fixed point results }
\begin{thm}
Let $(X,\rho_{r})$ be a $\rho_{r}-$complete rectangular metric like space, and $T$ a self mapping on $X.$
If there exists $0<k<1$ such that
 $$\rho_{r}(Tx,Ty)\le k \rho_{r}(x,y) \ \text{ for all } \ x,y \in X, \ \ \ (1)$$
then $T$ has a unique fixed point $u$ in $X,$ where $\rho_{r}(u,u)=0.$
\end{thm}
\begin{proof}
Let $x_{0}\in X$ and define the sequence $\{x_{n}\}$ by
$$x_{1}=Tx_{0}, x_{2}=Tx_{1}=T^{2}x_{0},\cdots, x_{n}=Tx_{n-1}=T^{n}x_{0},\cdots$$
Note that, if there exists a natural number $n$ such that $\rho_{r}(x_{n},x_{n+1})=0,$
then $x_{n}=x_{n+1}$ which implies that $x_{n}$ is a fixed point of $T$ and we are done.
Also, if $x_{n}=x_{n+1}$ for some $n,$ then $x_{n}$ is the fixed point of $T$ and we also done.
So, we may assume that $\rho_{r}(x_{n},x_{n+1})>0,$ and $x_{n}\neq x_{n+1}$ for all $n.$

First, consider the following notations: $$\rho_{n}=\rho_{r}(x_{n},x_{n+1}), \ \rho^{*}_{n}=\rho_{r}(x_{n},x_{n+2}) \ \text{ and} \ \rho^{'}_{n}=\rho_{r}(x_{n},x_{n}).$$
Hence,
\begin{align*}
\rho^{'}_{n}=\rho_{r}(x_{n},x_{n})=\rho_{r}(Tx_{n-1},Tx_{n-1})&\le k \max\{ \rho_{r}(x_{n-1},x_{n-1}),\rho_{r}(x_{n-1},x_{n-1}),\rho_{r}(x_{n-1},x_{n-1})\}\\
&=k \rho^{'}_{n-1}\\ & \le \cdots \\ &\le k^{n} \rho^{'}_{0} .
\end{align*}
Thereby, $$\lim_{n\rightarrow \infty}\rho^{'}_{n}=0$$
Also, by using $(1)$ we obtain:
\begin{align*}
\rho_{n}=\rho_{r}(x_{n},x_{n+1})=\rho_{r}(Tx_{n-1},Tx_{n})&\le k \rho_{r}(x_{n-1},x_{n})=\rho_{n-1}\\
&\le k^{2}\rho_{n-2}\\ &\le \cdots\\ &\le k^{n}\rho_{0}.
\end{align*}
Therefore, $$\rho_{n}\le k^{n}\rho_{0}. \ \ \ (2)$$
Similarly, it not difficult to see that
 $$\rho^{*}_{n}\le k^{n}\rho^{*}_{0}. \ \ \ (3)$$
Now, if for some $n> 0$ we have $x_{0}=x_{n},$ then
\begin{align*}
\rho_{0}&=\rho_{r}(x_{0},Tx_{0})\\ &=\rho_{r}(x_{n},Tx_{n})\\ &=\rho_{n}\\ &\le k^{n}\rho_{0},
\end{align*}
which leads to a contradiction. Thus, in this case we have $\rho_{0}=0$ and that is $x_{0}=x_{1},$ therefore
$x_{0}$ is a fixed point of $T.$ Thus, we may assume now that $x_{n}\neq x_{m}$ for all natural numbers $n\neq m.$
Next, we claim that $\rho_{r}(x_{n},x_{n+p})\rightarrow 0$ as $n,p\rightarrow \infty.$
To prove the claim we need to consider the following two cases:

\textbf{Case 1:} $p=2m+1$ (i.e: $p$ is odd). Hence, by $(1)$ and $(2)$ we have:
\begin{align*}
\rho_{r}(x_{n},x_{n+2m+1})&\le \rho_{r}(x_{n},x_{n+1})+\rho_{r}(x_{n+1},x_{n+2})+\rho_{r}(x_{n+2},x_{n+2m+1})\\&\le
\cdots \\ & \le \rho_{r}(x_{n},x_{n+1})+\rho_{r}(x_{n+1},x_{n+2})+\cdots +\rho_{r}(x_{n+2m},x_{n+2m+1})\\
&\le k^{n}\rho_{0}+k^{n+1}\rho_{0}+\cdots+k^{n+2m}\rho_{0}\\
&= k^{n}\rho_{0}[1+k+k^{2}+\cdots+k^{2m}]\\
&= k^{n}\rho_{0}(\displaystyle\frac{1-k^{2m+1}}{1-k}).
\end{align*}
Taking the limit in above inequality we obtain:
$$\rho_{r}(x_{n},x_{n+2m+1})\rightarrow 0 \ \text{as} \ n,m\rightarrow \infty$$

\textbf{Case 2:} $p=2m$ (i.e: $p$ is even). Hence, by $(1),$ $(2)$ and $(3)$ we have:
\begin{align*}
\rho_{r}(x_{n},x_{n+2m})&\le \rho_{r}(x_{n},x_{n+1})+\rho_{r}(x_{n+1},x_{n+2})+\rho_{r}(x_{n+2},x_{n+2m})\\&\le
\cdots \\ & \le \rho_{r}(x_{n},x_{n+1})+\rho_{r}(x_{n+1},x_{n+2})+\cdots +\rho_{r}(x_{n+2m-3},x_{n+2m-2})+\rho_{r}(x_{n+2m-2},x_{n+2m})\\
&\le k^{n}\rho_{0}+k^{n+1}\rho_{0}+\cdots+k^{n+2m-3}\rho_{0}+k^{n+2m-2}\rho^{*}_{0}
\end{align*}
Using the fact that $0<k<1$ and taking the limit in above inequality we obtain:
$$\rho_{r}(x_{n},x_{n+2m})\rightarrow 0 \ \text{as} \ n,m\rightarrow \infty$$

Thus, $\{x_{n}\}$ is a $\rho_{r}-$Cauchy sequence. Since $(X,\rho_{r})$ is a $\rho_{r}-$complete rectangular metric like space,
we deduce that $\{x_{n}\}$ converges to some $u\in X,$ such that
$$\lim_{n\rightarrow \infty}\rho_{r}(x_{n},x_{n})=\rho_{r}(u,u)$$
On the other hand, we have $$0= \lim_{n\rightarrow \infty}\rho^{'}_{n}=\lim_{n\rightarrow \infty}\rho_{r}(x_{n},x_{n})=\rho_{r}(u,u)$$
Hence, $$\rho_{r}(u,u)=0$$
Now, \begin{align*}
 \rho_{r}(u,Tu)&\le\rho_{r}(u,x_{n})+\rho_{r}(x_{n},x_{n+1})+\rho_{r}(x_{n+1},Tu)\\ &\le
 \rho_{r}(u,x_{n})+\rho_{r}(x_{n},x_{n+1})+k\rho_{r}(x_{n},u).
 \end{align*}
Taking the limit as $n\rightarrow \infty$ we deduce that $\rho_{r}(u,Tu)=0$ and that is $Tu=u.$ Therefore, $T$ has a fixed point in $X.$
To show the uniqueness of the fixed point, assume that $T$ has two fixed points
say $u$ and $v,$ hence

$$\rho_{r}(u,v)=\rho_{r}(Tu,Tv)\le k \rho_{r}(u,v)< \rho_{r}(u,v).$$ Thus,
$$\rho_{r}(u,v)=0,$$ which implies that $u=v$ as required.
\end{proof}

\begin{thm}
Let $(X,\rho_{r})$ be a $\rho_{r}-$complete rectangular metric like space, and $T$ a self mapping on $X.$
If there exists $0<k<1$ such that
 $$\rho_{r}(Tx,Ty)\le k \max\{\rho_{r}(x,y),\rho_{r}(x,x),\rho_{r}(y,y)\} \ \text{ for all } \ x,y \in X, \ \ \ \ (4)$$
then $T$ has a unique fixed point $u$ in $X,$ Where $\rho_{r}(u,u)=0$
\end{thm}
\begin{proof}
Let $x_{0}\in X$ and define the sequence $\{x_{n}\}$ by
$$x_{1}=Tx_{0}, x_{2}=Tx_{1}=T^{2}x_{0},\cdots, x_{n}=Tx_{n-1}=T^{n}x_{0},\cdots$$
Note that, if there exists a natural number $n$ such that $\rho_{r}(x_{n},x_{n+1})=0,$
then $x_{n}=x_{n+1}$ which implies that $x_{n}$ is a fixed point of $T$ and we are done.
Also, if $x_{n}=x_{n+1}$ for some $n,$ then $x_{n}$ is the fixed point of $T$ and we also done.
So, we may assume that $\rho_{r}(x_{n},x_{n+1})>0,$ and $x_{n}\neq x_{n+1}$ for all $n.$

First, consider the following notations: $$\rho_{n}=\rho_{r}(x_{n},x_{n+1}), \ \rho^{*}_{n}=\rho_{r}(x_{n},x_{n+2}) \ \text{ and} \ \rho^{'}_{n}=\rho_{r}(x_{n},x_{n}).$$

Hence,
\begin{align*}
\rho^{'}_{n}=\rho_{r}(x_{n},x_{n})=\rho_{r}(Tx_{n-1},Tx_{n-1})&\le k \max\{ \rho_{r}(x_{n-1},x_{n-1}),\rho_{r}(x_{n-1},x_{n-1}),\rho_{r}(x_{n-1},x_{n-1})\}\\
&=k \rho^{'}_{n-1}\\ & \le \cdots \\ &\le k^{n} \rho^{'}_{0} .
\end{align*}
Thus, $\{\rho_{r}(x_{n},x_{n})\}$ is a decreasing sequence, and
$$\rho^{'}_{n}\rightarrow 0 \ \text{as} \ n\rightarrow \infty. \ \ \ (5)$$

By $(4)$ we have
$$\rho_{n}=\rho_{r}(x_{n},x_{n+1})=\rho_{r}(Tx_{n-1},Tx_{n})\le k \max\{ \rho_{r}(x_{n-1},x_{n}),\rho_{r}(x_{n},x_{n}),\rho_{r}(x_{n-1},x_{n-1})\}$$
Given the fact that $\{\rho_{r}(x_{n},x_{n})\}$ is a decreasing sequence we have two cases in the above inequality.\\
\textbf{Case 1:} $\max\{ \rho_{r}(x_{n-1},x_{n}),\rho_{r}(x_{n-1},x_{n-1})\}=\rho_{r}(x_{n-1},x_{n-1}),$ in this case
and by $(5)$ we deduce that $$\rho_{r}(x_{n},x_{n+1})\rightarrow 0 \ \text{as} \ n\rightarrow \infty.$$
\textbf{Case 2:} $\max\{ \rho_{r}(x_{n-1},x_{n})\rho_{r}(x_{n-1},x_{n-1})\}=\rho_{r}(x_{n-1},x_{n}),$
note that if there exists $i<n$ such that $\max\{ \rho_{r}(x_{i},x_{n}),\rho_{r}(x_{i},x_{i})\}=\rho_{r}(x_{i},x_{i}).$ Then
by using $(4)$ repeatedly until we reach $i$ we get case 1 and in this case one can easily deduce that
$\rho_{r}(x_{n},x_{n+1})\rightarrow 0 \ \text{as} \ n\rightarrow \infty.$ So we may assume that
$\max\{ \rho_{r}(x_{i},x_{n}),\rho_{r}(x_{i},x_{i})\}=\rho_{r}(x_{i},x_{n}),$ for all $i<n.$ Therefore,
 \begin{align*}
\rho_{r}(x_{n},x_{n+1})=\rho_{r}(Tx_{n-1},Tx_{n})&\le k \rho_{r}(x_{n-1},x_{n})\\
&\le k^{2}\rho_{r}(x_{n-2},x_{n-1})\\ &\le \cdots\\ &\le k^{n}\rho_{r}(x_{0},x_{1}).
\end{align*}
Hence, since $0<k<1$ we deduce that $\rho_{r}(x_{n},x_{n+1})\rightarrow 0 \ \text{as} \ n\rightarrow \infty.$

Thus, in both cases we have  $$\rho_{n}\rightarrow 0 \ \text{as} \ n\rightarrow \infty. \ \ (6)$$
Next, note that by $(4)$ we have:
$$\rho^{*}_{n}=\rho_{r}(x_{n},x_{n+2})=\rho_{r}(Tx_{n-1},Tx_{n+1})\le k \max\{ \rho_{r}(x_{n-1},x_{n+1}),\rho_{r}(x_{n+1},x_{n+1}),\rho_{r}(x_{n-1},x_{n-1})\}$$
Given the fact that $\{\rho_{r}(x_{n},x_{n})\}$ is a decreasing sequence we have two cases in the above inequality.\\
\textbf{Case 1:} $\max\{ \rho_{r}(x_{n-1},x_{n+1}),\rho_{r}(x_{n-1},x_{n-1})\}=\rho_{r}(x_{n-1},x_{n-1}),$ in this case
and by $(5)$ we deduce that $$\rho_{r}(x_{n},x_{n+2})\rightarrow 0 \ \text{as} \ n\rightarrow \infty.$$
\textbf{Case 2:} $\max\{ \rho_{r}(x_{n-1},x_{n+1})\rho_{r}(x_{n-1},x_{n-1})\}=\rho_{r}(x_{n-1},x_{n+1}),$
note that if there exists $i<n$ such that $\max\{ \rho_{r}(x_{i},x_{n}),\rho_{r}(x_{i},x_{i})\}=\rho_{r}(x_{i},x_{i}).$ Then
by using $(4)$ repeatedly until we reach $i$ we get case 1 and in this case one can easily deduce that
$\rho_{r}(x_{n},x_{n+2})\rightarrow 0 \ \text{as} \ n\rightarrow \infty.$ So we may assume that
$\max\{ \rho_{r}(x_{i},x_{n}),\rho_{r}(x_{i},x_{i})\}=\rho_{r}(x_{i},x_{n}),$ for all $i<n.$ Therefore,
 \begin{align*}
\rho_{r}(x_{n-1},x_{n+1})=\rho_{r}(Tx_{n-2},Tx_{n})&\le k \rho_{r}(x_{n-2},x_{n})\\
&\le k^{2}\rho_{r}(x_{n-3},x_{n-1})\\ &\le \cdots\\ &\le k^{n-2}\rho_{r}(x_{0},x_{2}).
\end{align*}
Hence, since $0<k<1$ we deduce that $\rho_{r}(x_{n},x_{n+2})\rightarrow 0 \ \text{as} \ n\rightarrow \infty.$

Thus, in both cases we have  $$\rho^{*}_{n}\rightarrow 0 \ \text{as} \ n\rightarrow \infty. \ \ (7)$$

Now, if for some $n> 0$ we have $x_{0}=x_{n},$ then
\begin{align*}
\rho_{0}&=\rho_{r}(x_{0},Tx_{0})\\ &=\rho_{r}(x_{n},Tx_{n})\\ &=\rho_{n},
\end{align*}
but, by $(6)$ we have $\rho_{n}\rightarrow 0 \ \text{as} \ n\rightarrow \infty.$ Thus, in this case we have $\rho_{0}=0$ and that is $x_{0}=x_{1},$ therefore $x_{0}$ is a fixed point of $T.$ Thereby, we may assume now that $x_{n}\neq x_{m}$ for all natural numbers $n\neq m.$
Similarly to the argument in the previous theorem, we claim that $\rho_{r}(x_{n},x_{n+p})\rightarrow 0$ as $n,p\rightarrow \infty.$
To prove the claim we need to consider the following two cases:

\textbf{Case 1:} $p=2m+1$ (i.e: $p$ is odd). Hence,
\begin{align*}
\rho_{r}(x_{n},x_{n+2m+1})&\le \rho_{r}(x_{n},x_{n+1})+\rho_{r}(x_{n+1},x_{n+2})+\rho_{r}(x_{n+2},x_{n+2m+1})\\&\le
\cdots \\ & \le \rho_{r}(x_{n},x_{n+1})+\rho_{r}(x_{n+1},x_{n+2})+\cdots +\rho_{r}(x_{n+2m},x_{n+2m+1}).
\end{align*}
Taking the limit in above inequality and by $(6)$ we obtain:
$$\rho_{r}(x_{n},x_{n+2m+1})\rightarrow 0 \ \text{as} \ n,m\rightarrow \infty$$

\textbf{Case 2:} $p=2m$ (i.e: $p$ is even). Thus,
\begin{align*}
\rho_{r}(x_{n},x_{n+2m})&\le \rho_{r}(x_{n},x_{n+1})+\rho_{r}(x_{n+1},x_{n+2})+\rho_{r}(x_{n+2},x_{n+2m})\\&\le
\cdots \\ & \le \rho_{r}(x_{n},x_{n+1})+\rho_{r}(x_{n+1},x_{n+2})+\cdots +\rho_{r}(x_{n+2m-3},x_{n+2m-2})+\rho_{r}(x_{n+2m-2},x_{n+2m})
\end{align*}
Since $0<k<1,$ taking the limit in above inequality by $(6)$ and $(7)$ we obtain:
$$\rho_{r}(x_{n},x_{n+2m})\rightarrow 0 \ \text{as} \ n,m\rightarrow \infty$$

Thus, $\{x_{n}\}$ is a $\rho_{r}-$Cauchy sequence. Since $(X,\rho_{r})$ is a $\rho_{r}-$complete rectangular metric like space,
we deduce that $\{x_{n}\}$ converges to some $u\in X,$ such that
$$\lim_{n\rightarrow \infty}\rho_{r}(x_{n},x_{n})=\rho_{r}(u,u)$$
On the other hand, we have $$0= \lim_{n\rightarrow \infty}\rho^{'}_{n}=\lim_{n\rightarrow \infty}\rho_{r}(x_{n},x_{n})=\rho_{r}(u,u)$$
Hence, $$\rho_{r}(u,u)=0$$
Now, \begin{align*}
 \rho_{r}(u,Tu)&\le\rho_{r}(u,x_{n})+\rho_{r}(x_{n},x_{n+1})+\rho_{r}(x_{n+1},Tu)\\ &\le
 \rho_{r}(u,x_{n})+\rho_{r}(x_{n},x_{n+1})+k\rho_{r}(x_{n},u).
 \end{align*}
Taking the limit as $n\rightarrow \infty$ we deduce that $\rho_{r}(u,Tu)=0$ and that is $Tu=u.$ Therefore, $T$ has a fixed point in $X.$

To show the uniqueness of the fixed point, first we want to bring the attention to the following fact,
If $v$ is a fixed point,
$$\rho_{r}(v,v)=\rho_{r}(Tv,Tv)\le k \rho_{r}(v,v)\le \cdots \le k^{n}\rho_{r}(v,v)< \rho_{r}(u,v).$$
Which implies that $\rho_{r}(v,v)=0.$
Now, assume that $T$ has two fixed points
say $u$ and $v.$ Since $\rho_{r}(u,u)=\rho_{r}(v,v)=0$ we can conclude that $\max\{\rho_{r}(u,v),\rho_{r}(u,u),\rho_{r}(v,v)\}=\rho_{r}(u,v).$ Thus,

$$\rho_{r}(u,v)=\rho_{r}(Tu,Tv)\le k \rho_{r}(u,v)< \rho_{r}(u,v).$$ Thus,
$$\rho_{r}(u,v)=0,$$ which implies that $u=v$ as desired.
\end{proof}

\begin{thm}
Let $(X,\rho_{r})$ be a $\rho_{r}-$complete rectangular metric like space, and $T$ a self mapping on $X.$
If there exists $0<k<1$ such that
 $$\rho_{r}(Tx,Ty)\le k \max\{\rho_{r}(x,y),\rho_{r}(x,Tx),\rho_{r}(y,Ty)\} \ \text{ for all } \ x,y \in X, \ \ \ (8)$$
then $T$ has a unique fixed point $u$ in $X,$ where $\rho_{r}(u,u)=0.$
\end{thm}

\begin{proof}
Let $x_{0}\in X$ and define the sequence $\{x_{n}\}$ by
$$x_{1}=Tx_{0}, x_{2}=Tx_{1}=T^{2}x_{0},\cdots, x_{n}=Tx_{n-1}=T^{n}x_{0},\cdots$$
Note that, if there exists a natural number $n$ such that $\rho_{r}(x_{n},x_{n+1})=0,$
then $x_{n}=x_{n+1}$ which implies that $x_{n}$ is a fixed point of $T$ and we are done.
Also, if $x_{n}=x_{n+1}$ for some $n,$ then $x_{n}$ is the fixed point of $T$ and we also done.
So, we may assume that $\rho_{r}(x_{n},x_{n+1})>0,$ and $x_{n}\neq x_{n+1}$ for all $n.$

First, consider the following notations: $$\rho_{n}=\rho_{r}(x_{n},x_{n+1}), \ \rho^{*}_{n}=\rho_{r}(x_{n},x_{n+2}) \ \text{ and} \ \rho^{'}_{n}=\rho_{r}(x_{n},x_{n}).$$
Hence,
for all natural number $n$ we have
$$\rho_{n}=\rho_{r}(x_{n},x_{n+1})=\rho_{r}(Tx_{n-1},Tx_{n})\le k\max \{ \rho_{r}(x_{n},x_{n+1}),\rho_{r}(x_{n-1},x_{n})\}. \ \ (9)$$
Hence, if $\max \{ \rho_{r}(x_{n},x_{n+1}),\rho_{r}(x_{n-1},x_{n})\}=\rho_{r}(x_{n},x_{n+1}),$ then inequality $(8)$ implies
$$\rho_{r}(x_{n},x_{n+1})< \rho_{r}(x_{n},x_{n+1})$$ which leads to a contradiction. Therefore,
$$\max \{ \rho_{r}(x_{n},x_{n+1}),\rho_{r}(x_{n-1},x_{n})\}=\rho_{r}(x_{n-1},x_{n}) \ \text{for all} \ n. \ \ (10)$$
Also, note that
\begin{align*}
\rho_{r}(x_{n},x_{n+1})=\rho_{r}(Tx_{n-1},Tx_{n})&\le k \rho_{r}(x_{n-1},x_{n})\\
&\le k^{2}\rho_{r}(x_{n-2},x_{n-1})\\ &\le \cdots\\ &\le k^{n}\rho_{r}(x_{0},x_{1}).
\end{align*}
Thus, since $0<k<1$ we deduce that $$\rho_{n}\rightarrow 0 \ \text{as} \ n\rightarrow \infty. \ \ \ (11)$$
On the other hand, we have

$$\rho^{'}_{n}=\rho_{r}(x_{n},x_{n})=\rho_{r}(Tx_{n-1},Tx_{n-1})\le k \max\{ \rho_{r}(x_{n-1},x_{n-1}),\rho_{r}(x_{n-1},x_{n})\},$$
in both cases, it is not difficult to conclude that
$$\lim_{n\rightarrow \infty}\rho^{'}_{n}=0$$
Similarly to the argument in proof of the previous theorem, one can easily deduce that
$$\lim_{n\rightarrow \infty}\rho^{*}_{n}=0$$
Now, if for some $n> 0$ we have $x_{0}=x_{n},$ then
\begin{align*}
\rho_{0}&=\rho_{r}(x_{0},Tx_{0})\\ &=\rho_{r}(x_{n},Tx_{n})\\ &=\rho_{n},
\end{align*}
but, by $(6)$ we have $\rho_{n}\rightarrow 0 \ \text{as} \ n\rightarrow \infty.$ Thus, in this case we have $\rho_{0}=0$ and that is $x_{0}=x_{1},$ therefore $x_{0}$ is a fixed point of $T.$ Thereby, we may assume now that $x_{n}\neq x_{m}$ for all natural numbers $n\neq m.$

Similarly to the argument in the previous theorem, we claim that $\rho_{r}(x_{n},x_{n+p})\rightarrow 0$ as $n,p\rightarrow \infty.$
To prove the claim we need to consider the following two cases:

\textbf{Case 1:} $p=2m+1$ (i.e: $p$ is odd). Hence,
\begin{align*}
\rho_{r}(x_{n},x_{n+2m+1})&\le \rho_{r}(x_{n},x_{n+1})+\rho_{r}(x_{n+1},x_{n+2})+\rho_{r}(x_{n+2},x_{n+2m+1})\\&\le
\cdots \\ & \le \rho_{r}(x_{n},x_{n+1})+\rho_{r}(x_{n+1},x_{n+2})+\cdots +\rho_{r}(x_{n+2m},x_{n+2m+1}).
\end{align*}
Taking the limit in above inequality and by $(6)$ we obtain:
$$\rho_{r}(x_{n},x_{n+2m+1})\rightarrow 0 \ \text{as} \ n,m\rightarrow \infty$$

\textbf{Case 2:} $p=2m$ (i.e: $p$ is even). Thus,
\begin{align*}
\rho_{r}(x_{n},x_{n+2m})&\le \rho_{r}(x_{n},x_{n+1})+\rho_{r}(x_{n+1},x_{n+2})+\rho_{r}(x_{n+2},x_{n+2m})\\&\le
\cdots \\ & \le \rho_{r}(x_{n},x_{n+1})+\rho_{r}(x_{n+1},x_{n+2})+\cdots +\rho_{r}(x_{n+2m-3},x_{n+2m-2})+\rho_{r}(x_{n+2m-2},x_{n+2m})
\end{align*}
Since $0<k<1,$ taking the limit in above inequality by $(6)$ and $(7)$ we obtain:
$$\rho_{r}(x_{n},x_{n+2m})\rightarrow 0 \ \text{as} \ n,m\rightarrow \infty$$

Thus, $\{x_{n}\}$ is a $\rho_{r}-$Cauchy sequence. Since $(X,\rho_{r})$ is a $\rho_{r}-$complete rectangular metric like space,
we deduce that $\{x_{n}\}$ converges to some $u\in X,$ such that
$$\lim_{n\rightarrow \infty}\rho_{r}(x_{n},x_{n})=\rho_{r}(u,u)$$
On the other hand, we have $$0= \lim_{n\rightarrow \infty}\rho^{'}_{n}=\lim_{n\rightarrow \infty}\rho_{r}(x_{n},x_{n})=\rho_{r}(u,u)$$
Hence, $$\rho_{r}(u,u)=0$$
Now, \begin{align*}
 \rho_{r}(u,Tu)&\le\rho_{r}(u,x_{n})+\rho_{r}(x_{n},x_{n+1})+\rho_{r}(x_{n+1},Tu)\\ &\le
 \rho_{r}(u,x_{n})+\rho_{r}(x_{n},x_{n+1})+k\rho_{r}(x_{n},u).
 \end{align*}
Taking the limit as $n\rightarrow \infty$ we deduce that $\rho_{r}(u,Tu)=0$ and that is $Tu=u.$ Therefore, $T$ has a fixed point in $X.$

To show the uniqueness of the fixed point, first we want to bring the attention to the following fact,
If $v$ is a fixed point, then
$$\rho_{r}(v,v)=\rho_{r}(Tv,Tv)\le k \rho_{r}(v,v)\le \cdots \le k^{n}\rho_{r}(v,v)< \rho_{r}(u,v).$$
Which implies that $\rho_{r}(v,v)=0.$
Now, assume that $T$ has two fixed points
say $u$ and $v.$ Since $\rho_{r}(u,u)=\rho_{r}(v,v)=0$ we can conclude that $\max\{\rho_{r}(u,v),\rho_{r}(u,u),\rho_{r}(v,v)\}=\rho_{r}(u,v).$ Thus,

$$\rho_{r}(u,v)=\rho_{r}(Tu,Tv)\le k \rho_{r}(u,v)< \rho_{r}(u,v).$$ Thus,
$$\rho_{r}(u,v)=0,$$ which implies that $u=v$ as desired.
\end{proof}

\section{Acknowledgements}
The  first three authors would like to thank Prince Sultan University for funding this work through research group
Nonlinear Analysis Methods in Applied Mathematics (NAMAM) group number RG-DES-2017-01-17.

\section{Competing interests}
The authors declare that they have no competing interests.

\section{Authors' contributions}
All the authors participated in obtaining the main results of this manuscript and drafted the manuscript. All authors
read and approved the final manuscript.

\end{document}